\documentclass[12pt,reqno]{amsart}
\usepackage{amssymb}
\usepackage{amsmath}
\usepackage{amsthm}
\usepackage{eucal}
\usepackage{color}
\usepackage{amsfonts}
\usepackage[pdftex]{graphicx}
\usepackage[notcite,notref]{showkeys}
\usepackage[T1]{fontenc}

\newcommand{\R}{\mathbb R}

\newcommand{\Z}{\mathbb Z}

\newcommand{\sgn}{\text{sgn}}

\numberwithin{equation}{section}
\newcommand{\be}{\begin{equation}}
\newcommand{\ee}{\end{equation}}
\newcommand{\bp}{\begin{proof}}
\newcommand{\ep}{\end{proof}}
\newcommand{\bel}{\begin{equation}\label}
\newcommand{\eeq}{\end{equation}}
\newtheorem{thm}{Theorem}[section]
\newtheorem{cor}[thm]{Corollary}
\newtheorem{lem}[thm]{Lemma}
\newtheorem{prop}[thm]{Proposition}

\theoremstyle{remark}
\newtheorem{rem}{Remark}[section]

\numberwithin{equation}{section}

\begin{document}
\title[BBM equation]{On special properties of solutions to the Benjamin-Bona-Mahony equation}
\author{C. Hong}
\address[C. Hong]{Department  of Mathematics\\
University of California\\
Santa Barbara, CA 93106\\
USA.}
\email{christianhong@ucsb.edu}

\author{G. Ponce}
\address[G. Ponce]{Department  of Mathematics\\
University of California\\
Santa Barbara, CA 93106\\
USA.}
\email{ponce@math.ucsb.edu}


\begin{abstract} 
This work is concerned with the  Benjamin-Bona-Mahony equation. This model was deduced as an approximation to the Korteweg-de Vries equation in the description  of the unidirectional propagation of long waves. Our goal here is to study unique continuation and regularity properties on  solutions to the associated initial value problem and intial periodic boundary value problems.
\end{abstract}
\maketitle

\section{Introduction and main results}


The Korteweg-de Vries (KdV) equation 
\be
\label{KdV}
\partial_t u +\partial_x u +u\partial_xu +\partial^3_{x}u = 0,
\ee
was first derived as a model of long waves in shallow water \cite{kdv}. In \cite{BBM} Benjamin-Bona-Mahony proposed the (BBM) equation 
\be
\label{bbm-c}
\partial_t u +\partial_x u +u\partial_xu -\partial^3_{xxt}u = 0,
\ee
as a regularized version  of the KdV equation. 

Thus, for the BBM equation  its accuracy as an approximation to the KdV, the existence and stability of traveling waves, their interaction, the well-posedness and their asymptotic behavior of the solutions among many other aspects  have been extensively regarded, see \cite{BPS}, \cite{Ol}, \cite{SS} and \cite{BoTz} and references therein.

Firstly, we shall be concerned with the  initial value problem (IVP) associated to the formally equivalent form of the BBM
 \be\label{bbm} 
\begin{aligned}
\begin{cases}
& \partial_t u +(1-\partial_x^2)^{-1}(\partial_x u +u\partial_xu) = 0,\qquad (t,x) \in \R\times \R,\\
&u(x,0)=u_0(x).
 \end{cases}
\end{aligned}
\ee
The solution of the IVP \eqref{bbm} can be expressed in the equivalent integral form :
\be\label{bbm1} 
u(x,t)= U(t)u_0(x)-\int_0^t U(t-t')\partial_xJ^{-2}\Big(\frac{u^2}{2}\Big)(x,t')dt',
\ee
where $\{U(t)\,:\,t\in \mathbb R\}$ denotes the unitary group 
\be
\label{group}
U(t)f(x)=e^{t \partial_xJ^{-2}}f(x)=(e^{-2\pi it\xi(1+4\pi^2 \xi^2)^{-1}} \widehat{f}(\xi))^{\vee}(x)
\ee
with
\be
\label{bessel}
J^sf(x) =(1-\partial_x^2)^{s/2}f(x)=((1+4\pi^2\xi^2)^{s/2} \widehat{f}(\xi))^{\vee}(x),\;\,s\in\mathbb R.
\ee

\vskip.1in

In \cite{BoTz}, it was shown that the IVP \eqref {bbm} is globally well-posed in classical Sobolev spaces
\be
\label{sob-L2}
H^s(\mathbb R)=J^{-s/2}L^2(\mathbb R)=(1-\partial_x^2)^{-s/2}L^2(\mathbb R)\;\;\text{with}\;\;\,s\geq 0.
\ee

\begin{thm}[\cite{BoTz}]\label{BoTz}
 Let $s\geq 0$. For any $u_0\in H^s(\mathbb R)$ and for any $\,T>0$ there exists a unique solution $u=u(x,t)$ of the IVP \eqref{bbm} satisfying
 \be
 \label{per-BT}
 u\in C([-T,T]:H^s(\mathbb R)).
 \ee
Moreover, the map data solution, $u_0\to u$, is locally analytic from $ H^s(\mathbb R)$ to $C([-T,T]:H^s(\mathbb R)).$
\end{thm}

This strong notion of well-posedness  includes,  existence, uniqueness, persistence properties, i.e   if $u_0\in X$ then the solution $u=u(x,t)$ describes a continuous curve on $X$ so
$u\in C([0,T]:X)$, and regularity of the solution upon the data.

In \cite{BoTz}, it was also proven that the result in Theorem \ref{BoTz} is the best possible in the scale of Sobolev spaces $H^s(\mathbb R)$. More precisely, it was verified that in $H^s(\mathbb R)$, for any $s<0$, the map data-solution $u_0\to u(x,t)$  is not $C^2$. Hence, well-posedness cannot be obtained  in $H^s(\mathbb R),\;s<0$, by an argument based on the contraction principle. This ill-posedness was strengthened in \cite{Pa}.

In \cite{Wa}, the well-posedness in $H^s(\mathbb R), \,s\geq 0$, established in \cite{BoTz} was extended to the $L^p$-Sobolev spaces  :
\be
\label{sob-Lp}
L^{p}_s(\mathbb R)=J^{-s/2}L^p(\mathbb R)=(1-\partial_x^2)^{-s/2}L^p(\mathbb R),\;\;p\in[1,\infty),
\ee
with
\be
\label{sob-rest}
s\geq \max\big\{\frac{1}{p}-\frac{1}{2};0\big\}\;\;\text{if}\;\;p\in(1,\infty),\;\;\;\text{and}\;\;\;s>\frac{1}{2},\;\;\text{if}\;\;p=1.
\ee
Moreover, it was shown in \cite{Wa} that the well-posedness results obtained are optimal in the scale described in \eqref{sob-Lp}-\eqref{sob-rest}.

For our purpose here we shall need a local well-posedness result for the IVP \eqref{bbm} in H\"older spaces, i.e. for $k\in {0,1,2,...} $ and $\theta\in [0,1]$
\be
\label{holder}
C^{k,\theta}(\mathbb R)=\{f:\mathbb R\to\mathbb R\,:\,f\in C^k\cap W^{k,\infty}(\mathbb R),\,\;\;f^{(k)}\in C^{\theta}(\mathbb R)\},
\ee
where for $k\in\mathbb Z, \,k\geq 0$,
\be
\label{spaceLinfinito}
W^{k,\infty}(\mathbb R)=\{f:\mathbb R\to \mathbb R\;:\, \partial^j_x f\in L^{\infty}(\mathbb R),\;j=0,1,..,k\},
\ee
with
\be
\label{holder-norm}
\|f\|_{C^{k,\theta}}=\sum_{j=0}^{k}\|f^{(j)}\|_{\infty}+ \sup_{h\neq 0}\,\Big\|\,\frac{f^{(k)}(\cdot+h)-f^{(k)}(\cdot)}{|h|^{\theta}}\Big\|_{\infty}.
\ee

\begin{thm}\label{THA1} 

Let $k\in \mathbb Z, k\geq 0$, and $\theta\in [0,1]$. For any $u_0\in C^{k,\theta}(\mathbb R)$ there exist $T=T(\|u_0\|_{C^{k,\theta}})>0$ and a unique solution $u=u(x,t)$ of the IVP \eqref{bbm} such that
\be
\label{class-holder}
u\in C^{\infty}((-T,T): C^{k,\theta}(\mathbb R))\equiv X_T^{k,\theta}(\mathbb R).
\ee
Moreover, the map data solution, $u_0\to u$, is locally analytic from $C^{k,\theta}(\mathbb R) $ to $X_T^{k,\theta}(\mathbb R).$
\end{thm}

We observe that the constant solution $u(x,t)\equiv c_0$ belongs to the class in \eqref{class-holder}. Here we shall not address the question on the possible global in time extension of the local solution obtained in Theorem \ref{THA1}.

 Next, we have our first result concerning unique continuation properties of solutions to the IVP \eqref{bbm}:

\begin{thm}\label{THA2} Let $u=u(x,t)$ be the solution of the IVP \eqref{bbm} provided by Theorem \ref{THA1}.
If there exist $t_0\in(-T,T) $ and $a,b\in \mathbb R$ with $ \,a<b$ such that 
\be 
\label{uc-cond1}
\begin{cases}
\begin{aligned}
&(i)\;\;\,\;u(x,t_0)=-1,\;\;\; \;\;x\in(a,b),\\
&(ii)\, \,\;\partial_tu(b,t_0)\geq\partial_tu(a,t_0),
\end{aligned}
\end{cases}
\ee
then $\,u(x,t)=-1$ for all $(x,t)\in \mathbb R\times (-T,T)$.
\end{thm}

\begin{cor}\label{noL2}
There does not exist any solution $u=u(x,t)$ of the IVP \eqref{bbm}  in the class \eqref{per-BT} provided by Theorem \ref{BoTz} satisfying the conditions in \eqref{uc-cond1}.

\end{cor} 

\begin{rem}

\begin{enumerate}
\item In Corollary \ref{noL2} the solution $$u\in C([-T,T]:H^s(\R))$$ satisfies \eqref{uc-cond1} (i), almost everywhere in $(a,b)$, if $s\in[0,1/2]$, 
and  \eqref{uc-cond1} (ii) pointwise for all $s\geq 0$, since in this case  $\partial_t u\in C(\R\times [-T,T])$.

\item The condition  \eqref{uc-cond1} is weaker than the usual one for unique continuation. That assumes the existence of a (non-empty) open set $\Omega\subset \mathbb R\times(-T,T)$ such that the solution $u=u(x,t)$ satisfies 
$u(x,t)=-1\;\text{if}\;(x,t)\in\Omega$.

\item For previous unique continuation results for the BBM, we refer  to \cite{RoZh} and references therein.

\item 
In Theorem \ref{THA2} the condition (ii) in \eqref{uc-cond1} can be restated as :
$$
\partial_tu(x,t_0)\;\;\;\;\text{is not strictly decreasing for }\;\;\;x\in[a,b].
$$
\item In the case of the KdV equation, see \eqref{KdV}, formally the condition  \eqref{uc-cond1} part (i),  and the equation  imply equality in  \eqref{uc-cond1} part  (ii).

\item The following example, found in \cite{RoZh},  shows that  Corollary \ref{noL2} fails for the value $-2$ in \eqref{uc-cond1} part (i). Let $a, b\in\R, \,a<b$, and
\be
\label{ex1}
u(x,t)=u_0(x)=
\begin{cases}
\begin{aligned} 
&\,-2,\;\;\;\;x\in[a,b],\\
&\;\;\;\,\;0,\;\;\;\;x\notin [a,b].
\end{aligned}
\end{cases}
\ee
Then $u\in C(\R:H^s(\R)), \,s\in[0,1/2)$,  solves the IVP \eqref{bbm} with $\partial_tu(x,t)=0,\,\forall \,(x,t)\in \R^2$

\end{enumerate}
\end{rem}

In fact,  following result shows that the value $-1$ in \eqref{uc-cond1} is the only one for which Theorem \ref{THA2} and Corollary \ref{noL2} hold.

\begin{thm}\label{THA3} Given any $c_0\in \mathbb R-\{-1\}$ and $a, b\in \R, a<b$, there exists $u_0\in C^{\infty}_0(\mathbb R)-\{0\}$ such that the corresponding solution $u=u(x,t)$ of the IVP \eqref{bbm} provided  by Theorem \ref{BoTz},  with $s>1/2$, (or  by Theorem \ref{THA1}) satisfies that
\be 
\label{uc-cond3}
(i)\,\;u_0(x)=c_0\;\;\;\;\text{and}\;\;\;\;(ii) \,\;\partial_tu(x,0)=0,\;\;\;x\in(a,b).
\ee
\end{thm}


With a similar approach one can get the following general result :

\begin{thm}\label{THA4}  Let $c_0\in\R$. Let  $u=u(x,t)$ be a solution of the IVP \eqref{bbm} provided   by Theorem \ref{THA1} such that for some $t_0\in (-T,T)$ and some $a, b\in\mathbb R, a<b$, one has
\be
\label{cond-aa1}
\,\;u(y,t_0)=c_0,\;\;y\in\{a,b\}.
\ee
If one of the following conditions hold with $f(y)=y+y^2/2$ :
\be
\label{cond-1}
\begin{cases}
\begin{aligned}
f(u(x,t_0))&\leq f(c_0) \;\;\;\text{if}\;\;\;x\in[a,b],\\
\;\;f(u(x,t_0))&\geq f(c_0)\;\;\;\text{if}\;\;\;x\notin [a,b],\\
\;\;\;\;\partial_tu(b,t_0)&\geq \partial_tu(a,t_0),
\end{aligned}
\end{cases}
\ee
or
\be
\label{cond-2}
\begin{cases}
\begin{aligned}
f(u(x,t_0))&\leq f(c_0)\; \;\;\text{if}\;\;\;x\notin[a,b],\\
\,\;f(u(x,t_0))&\geq f(c_0)\;\;\;\text{if}\;\;\,x\in [a,b],\\
\;\;\;\;\partial_tu(a,t_0)&\geq \partial_tu(b,t_0),
\end{aligned}
\end{cases}
\ee
then $u(x,t)\equiv c_0$.

\end{thm}

\begin{rem}

\begin{enumerate}
\item In the case $c_0=-1$, since  $f(y)=y+y^2/2$ realizes its minimum at $y=-1$, one sees that the condition \eqref{cond-1} agrees with \eqref{uc-cond1}  in   Theorem \ref{THA2}.

\item For $c_0=0$, using that $f(0)=0$, from Theorem \ref{THA4}, one obtains the following conclusion  : if the solution  $u(x,t)$ of the IVP \eqref{bbm}, provided by Theorem \ref{THA2}, satisfies
\be
\label{cond-aa2}
\;u(y,t_0)=\partial_tu(y,t_0)=0,\;\;y\in\{a,b\},
\ee
for some $t_0\in(-T,T)$ and $a<b$, with  $u(x,t_0)$ having different sign in $[a,b]$ and $[a,b]^c$, then $u(x,t)$ is identically zero.  In this case, 
the solution $u(x,t)$ belongs to the class described in Theorem \ref{BoTz}.

\item As we will see, Theorem \ref{RoZh1} below, previous unique continuation results assume some non-local (although non-pointwise) conditions on the solution.

\end{enumerate}
\end{rem}

Next, we consider the IPBVP associated to the BBM equation
 \be
 \label{bbm-p} 
\begin{aligned}
\begin{cases}
& \partial_t u +(1-\partial_x^2)^{-1}(\partial_x u +u\partial_xu) = 0,\qquad (t,x) \in \R\times \mathbb S,\\
&u(x,0)=u_0(x).
 \end{cases}
\end{aligned}
\ee

In this regards we have the following local well-posedness result established in \cite{RoZh} (see also \cite{BoTz} and  \cite{Rou}):

\begin{thm}[\cite{RoZh}]\label{RoZh}
 Let $s\geq 0$. For any $u_0\in H^s(\mathbb S)$ and for any $\,T>0$ there exists a unique solution $u=u(x,t)$ of the IPBVP \eqref{bbm-p} such that
 \be
 \label{per-BT1}
 u\in C([-T,T]:H^s(\mathbb S)).
 \ee
Furthemore, the map data solution, $u_0\to u$, is locally real analytic from $ H^s(\mathbb R)$ to $C([-T,T]:H^s(\mathbb S)).$
\end{thm}

It was also shown in \cite{RoZh} that if $u_0\in H^1(\mathbb S)$ the quantities
\be
\label{cl}
\int_{\mathbb S}u(x,t)dx,\;\;\int_{\mathbb S}((\partial_xu)^2+u^2)(x,t)dx,\;\;\int_{\mathbb S}(u^3+3u^2)(x,t)dx
\ee
are time independent. This was used in the proof of the next result achieved in \cite{RoZh}:

\begin{thm}[\cite{RoZh}]\label{RoZh1}
 Let  $u_0\in H^1(\mathbb S)$ be such that
 \be
 \label{hypp1}
 \int_{\mathbb S}u_0(x)dx\geq 0,
 \ee
 and
 \be
 \label{hypp2}
\| u_0\|_{\infty}<3.
 \ee
Assume that the corresponding solution $u=u(x,t)$ of the IPBVP \eqref{bbm-p} provided by Theorem \ref{RoZh} satisfies
 \be
 \label{hypp3}
 u(x,t)=0\;\;\;\;\;\;\forall\,(x,t)\in \omega\times(0,T),
 \ee
 where $\omega\subset \mathbb S$ is a non-empty open set.  Then $u(x,t)\equiv 0$.
\end{thm}

\begin{rem}

\begin{enumerate}

\item
Notice that Theorem \ref{RoZh1} contain the global hypothesis \eqref{hypp1} and \eqref{hypp2}. 
\item
The proof used the conservation laws stated in \eqref{cl}. This explains the hypothesis $u_0\in H^1(\mathbb S)$.

\item As it was remarked in \cite{RoZh}, the example in \eqref{ex1}  shows that Theorem \ref{RoZh1} fails if one just assumes that $u_0\in H^s(\mathbb S)\cap L^{\infty}(\mathbb S),\,s\in[0,1/2)$.
\end{enumerate}
\end{rem}
Our results in Theorem \ref{THA2} and Theorem \ref{THA3} extend to the IPBVP \eqref{bbm-p}:

\begin{thm}\label{THAP2} Let $u=u(x,t)$ be a solution of the IPBVP \eqref{bbm-p} provided by Theorem \ref{RoZh}.
If there exist $t_0\in(-T,T) $ and $a,b\in \mathbb S\cong[0,1)$ with $ \,a<b$ such that 
\be 
\label{uc-cond1p}
\begin{cases}
\begin{aligned}
&(i)\;\;\,\;u(x,t_0)=-1,\;\;\;a.e. \;\;x\in(a,b),\\
&(ii)\, \,\;\partial_tu(b,t_0)\geq\partial_tu(a,t_0),
\end{aligned}
\end{cases}
\ee
then $\,u(x,t)=-1$ for all $(x,t)\in \mathbb S\times [-T,T]$,  if $s>1/2$, and a.e. for $(x,t)\in \mathbb S\times [-T,T]$, if $s\in [0,1/2]$.
\end{thm}

\begin{rem}

One can easily see that if the solution $u\in C([-T,T]:L^2(\mathbb S))$, then $\partial_t u\in C(\mathbb S\times [-T,T])$. Hence, the pointwise condition \eqref{uc-cond1p} is well-defined.

The following theorem and its remark  confirm that the value $-1$ is the only one for which  Theorem \ref{THAP2} holds.
\end{rem}
\begin{thm}\label{THAP3} Given any $c_0\in \mathbb R-\{-1\}$ there exist $a, b\in [0,1)\cong \mathbb S$ and  $u_0\in C^{\infty}(\mathbb S)-\{c_0\}$ such that the corresponding solution $u=u(x,t)$ of the IPBVP \eqref{bbm-p} provided  by Theorem \ref{RoZh} satisfies that
\be 
\label{uc-cond3p}
\,\;u_0(x)=c_0,\;\;\;\;x\in(a,b)\subset [0,1)\cong \mathbb S,
\ee
and
\be
\label{uc-cond4p}
 \;\partial_tu(a,0)=\partial_tu(b,0).
\ee
\end{thm}

\begin{rem}

The following example is an extension of that attained in  \cite{RoZh}. Let $f(x)=x+x^2/2$, $c_0\in \R-\{-1\}$ and $a, b\in (0,1)\cong \mathbb S, a< b$.  Let 
$$
f^{-1}(f(c_0))=\{c_0,c_0'\},\;\;\;\;\;c_0\neq c_0'.
$$ 
Then
\be
\label{zz1}
u(x,t)=u_0(x)=
\begin{cases}
\begin{aligned}
& c_0.\;\;\;\;x\in[a,b],\\
&c_0',\;\;\;\,x\notin [a,b].
\end{aligned}
\end{cases}
\ee
is a solution of the IPBVP \eqref{bbm-p}, with  $u\in C([-T,T]:H^s(\mathbb S))$, $s\in[0,1/2)$, and $\partial_tu(x,t)\equiv 0$.

\end{rem}

Also we have: 

\begin{thm}\label{THAP4} Under the appropriate modifications the results in Theorem \ref{THA4} still hold for solutions of the IPBVP \eqref{bbm-p} provided by Theorem \ref{RoZh} with $s>1/2$.

\end{thm}

The argument of proof of these results allows us to get a slight improvement of those found in \cite{LiPo} concerning the Camassa-Holm (CH)  equation. The CH equation was first observed by Fuchssteiner and Fokas \cite{FF} in their work on hereditary symmetries. Later, it was written explicitly and derived physically as a model for shallow water waves  by Camassa and Holm \cite{CH}. The CH equation has received considerable attention due to its remarkable properties, among them the fact that it is a bi-Hamiltonian completely integrable model and the existence of \lq\lq peakon'' solutions (solitons)
\be
\label{peakon}
u_c(x,t)=c\, e^{-|x-ct|},\;\;\;\;\;\;c>0,
\ee
(see \cite{LiPo} and references therein).

The CH equation can be written as
\begin{equation}\label{CH}
\partial_tu+ u \partial_xu+\partial_x(1-\partial_x^2)^{-1}\big(u^2+\frac{1}{2} (\partial_xu)^2\big)=0, \hskip5pt \;t,\,x\in\R.
\end{equation}

For our purpose here, we recall the following local "well-posedness" for the IVP associated to the CH equation \eqref{CH} obtained in \cite{LiPoSi}:

\begin{thm}[\cite{LiPoSi}]\label{thm1-ch}
Given  $u_0\in X\equiv H^1(\R)\cap W^{1,\infty}(\R)$, there exist  $T=T(\|u_0\|_X)>0$ and a unique solution $u=u(x,t)$ of the IVP associated to the CH equation \eqref{CH}
such that
\begin{equation}
\begin{aligned}
\label{class-sol}
u\in
&C([-T,T]\!:\!H^1(\R))\cap C^1((-T,T)\!:\!L^2(\R))\\
&\cap L^{\infty}([-T,T]\!:\!W^{1,\infty}(\R))\equiv Z_T\cap L^{\infty}([-T,T]\!:\!W^{1,\infty}(\R)).
\end{aligned}
\end{equation}
Moreover, the map $u_0\mapsto u$, taking the data to the solution, is locally continuous from $X$
into $Z_{T}$.
\end{thm}

We notice that the peakons belong to the space $X$ in Theorem \ref{thm1-ch}. However, the notion of well-posedness here is not as strong as those stated in previous theorems. For other \lq\lq well-posedness" results concerning the IVP \eqref{CH} we refer to \cite{LiPoSi} and references therein.

 In \cite{LiPo} this unique continuation principle was derived :
\begin{thm}[\cite{LiPo}] \label{IVPCH1}
Let $\,u=u(x,t)$ be a solution of the IVP associated to the CH equation \eqref{CH} in the class described in Theorem \ref{thm1-ch}.
If there exists an open set $\,\Omega\subset \R\times [0,T]$ such that 
\begin{equation}
\label{cond-ch}
u(x,t)=0,\;\;\;\;\;\;\;(x,t)\in \Omega,
\end{equation}
then $\,u\equiv 0$.
\end{thm}

Our argument here shows the following:
\begin{thm} \label{IVPCH2}
Let $\,u=u(x,t)$ be a solution of the IVP associated to the CH equation \eqref{CH} in the class described in Theorem \ref{thm1-ch}.
If there exist $t_0\in(-T,T)$ and $a, b\in\R, a<b$, such that
\begin{equation}
\label{cond1-ch}
u(x,t_0)=0,\;\;\;\;\;\;\;x\in[a,b],
\end{equation}
and
\be
\label{cond2-ch}
\partial_tu(x,t_0)\;\;\;\;\text{is not strictly decreasing on}\;\;\;\;[a,b],
\end{equation}
then $\,u\equiv 0$.
\end{thm}

One observes that from hypothesis \eqref{cond1-ch} it follows that $\partial_tu(\cdot,t_0)$ is continuous in $[a,b]$ such that the pointwise evaluation in \eqref{cond2-ch} makes sense. 

We remark that the result in Theorem \ref{IVPCH2} expands to all the b-equations (see \cite{LiPo}) for both their IVP and their IPBVP.

A main tool in the proof of the above unique continuation results is a refinement in the application of a lemma established in \cite{LiPo}, (see Lemma \ref{key-uc} below).

Next, we shall examine some special regularity properties of solutions to the BBM equation.

In \cite{BoTz} and \cite{Wa} the integral form of the equation written in \eqref{bbm1} was considered. Here, we shall work with the following equivalent integral version (for the IVP and the IPBVP) of \eqref{bbm} :
\be\label{bbm2} 
u(x,t)=u_0(x)-\int_0^t \partial_xJ^{-2} \big(u+\frac{u^2}{2}\big)(x,t')dt'.
\ee

To simplify our exposition in the case of the IVP we shall restrict ourselves to consider only solutions provided by Theorem \ref{BoTz} and Theorem \ref{THA1}. It will be clear that the results here  extend to the $L^p_s(\mathbb R)$-solutions obtained in \cite{Wa} .

In this direction we have:

\begin{thm}\label{TH1}

(i) Let $u_0\in H^{s}(\mathbb R),\;s\geq 0$. Let $$u\in C([-T,T]:H^s(\mathbb R))$$ be the  corresponding solution of the IVP \eqref{bbm}  provided by Theorem \ref{BoTz}. Then
\be
 \label{m5}
u-u_0\in C^{\infty}([-T,T]:H^{(2s+1/2)^-}(\R)),\;\;\;\text{if}\;\;\;s\in [0,1/2],
\ee
and
\be
 \label{m6}
u-u_0\in C^{\infty}([-T,T]:H^{s+1}(\R)),\;\;\;\;\text{if}\;\;\;\;s>1/2.
\ee

(ii)  Let $u_0\in C^{k,\theta}(\mathbb R),\;k\in\mathbb Z,\,k\geq 0,\;\theta\in[0,1]$. Let $$u\in C([-T,T]: C^{k,\theta}(\mathbb R))$$ be the  corresponding solution of the IVP \eqref{bbm}  provided by Theorem \ref{THA1}. Then
\be
 \label{m5a}
u-u_0\in C^{\infty}([-T,T]: C^{k+1,\theta}(\mathbb R)).
\ee

\end{thm}

\begin{thm}\label{TH1b} The results in Theorem \ref{TH1} part (i) holds for solutions of the IPBVP provided by Theorem \ref{RoZh}.

\end{thm}

\begin{rem}

\begin{enumerate}

\item Concerning the time regularity,  our results  show that the given solution $u\in C(\mathbb R:X)$  belongs to $C^{\infty}(\mathbb R:X)$.

In particular, for the explicit solutions described in \eqref{ex1} and \eqref{zz1}, one has that  $u\in C(\mathbb R:H^2(\mathcal T)), \,s\in[0,1/2)$, for $\mathcal T=\R$ or $\mathbb S$, with $\partial_tu(x,t)=u(x,t)-u_0(x)\equiv 0,\;(x,t)\in  \mathcal T\times \R$.

\item The arguments below also show that the group $\{U(t)\,:\,t\in\mathbb R\}$ defined in \eqref{group} satisfies: if $u_0\in H^s(\mathbb R)$, then 
$$
U(t)u_0-u_0\in C^{\infty}(\mathbb R: H^{s+1}(\mathbb R)).
$$
In the case $u_0\in C^{k,\theta}(\mathbb R)$, one has 
$$
\partial_x(U(t)u_0-u_0)\in C^{\infty}(\mathbb R:C^{k+1,\theta}(\mathbb R)).
$$

\item In Theorem \ref{BoTz} and Theorem \ref{RoZh}  one can use the continuous dependence of the solution upon the data  to see that the solutions of \eqref{bbm-c} are the  limit of classical solutions of the  equations \eqref{bbm} and \eqref{bbm-c}. In the case $k\in\mathbb Z, k\geq 0$ with  $\theta=0$ the solutions in
 Theorem  \ref{THA1}   satisfy the same property. However, this does not hold when $k=0,1,2$ and  $\theta\in(0,1]$, since  in this setting  the intitial data is not the limit of smooth functions.

\end{enumerate}
\end{rem}

Next, we recall the definition of $H^s(\Omega)$, $s\geq 0$, and $\Omega\subset \mathbb R$ open set. Given $f:\Omega\to\mathbb R$ one says that $f\in H^s(\Omega)$ if there exists $F\in H^s(\mathbb R)$ such that $F\big|_{\Omega}=f$. The norm of $f$ is the infimum of the norm in $H^s(\mathbb R)$ among  all the possible extensions. Similarly for 
$C^{k,\theta}(\Omega)$.

\begin{thm}\label{TH1c} 

Let $u_0\in H^s(\mathbb R), \,s\geq 0,$ and $\Omega\subseteq \mathbb R$ be an open set.

(a) If $s\leq 1/2$ and $u_0\in H^{s+\eta}(\Omega)$, $\eta\in(0,s+1/2)$, then the corresponding solution $u$ of the IVP \eqref{bbm} provided by Theorem \ref{BoTz} satisfies:
\be
\label{z1}
u\big|_{\Omega}\in C(\mathbb R:H^{s+\eta}(\Omega)).
\ee

(b) If $s\leq 1/2$ and $u_0\notin H^{s+\eta}(\Omega)$, $\eta\in(0,s+1/2)$, then the corresponding solution $u$ of the IVP \eqref{bbm} provided by Theorem \ref{BoTz}  satisfies:
\be
\label{z2}
u\big|_{\Omega}\notin H^{s+\eta}(\Omega),\;\;\;\;\forall\,t\in \,\mathbb R.
\ee

(c) If $s>1/2$ and $u_0\in H^{s+\eta}(\Omega)$, $\eta\in(0,1]$, then the corresponding solution $u$ of the IVP \eqref{bbm}  provided by Theorem \ref{BoTz} satisfies:
\be
\label{z3}
u\big|_{\Omega}\in C(\mathbb R:H^{s+\eta}(\Omega)).
\ee

(d) If $s>1/2$ and $u_0\notin H^{s+\eta}(\Omega)$, $\eta\in(0,1]$, then the corresponding solution $u$ of the IVP \eqref{bbm} provided by Theorem \ref{BoTz} satisfies:
\be
\label{z4}
u\big|_{\Omega}\notin H^{s+\eta}(\Omega),\;\;\;\;\forall\,t\in \,\mathbb R.
\ee

\end{thm}

\begin{thm}\label{TH1d} 

Let $u_0\in C^{k,\theta}(\mathbb R), \,k\in\mathbb Z, k\geq 0,\;\theta\in[0,1]$ and $\Omega\subseteq \mathbb R$ be an open set.

(a) If $u_0\in C^{j,\eta}(\Omega)$, $j\in\{k;k+1\}$, $\eta\in[0,1]$, and $j+\eta\leq k+1+\theta$, then the corresponding solution $u$ of the IVP \eqref{bbm} provided by Theorem \ref{THA1} satisfies:
\be
\label{z5}
u\big|_{\Omega}\in C^{\infty}(\mathbb R:C^{j,\eta}(\Omega)).
\ee

(b)  If $u_0\notin C^{j,\eta}(\Omega)$, $j\in\{k;k+1\}$, $\eta\in[0,1]$, and $j+\eta\leq k+1+\theta$, then the corresponding solution $u$ of the IVP \eqref{bbm} provided by Theorem \ref{THA1}  satisfies:
\be
\label{z6}
u\big|_{\Omega}\notin C(\mathbb R:C^{j,\eta}(\Omega)),\;\;\;\;\forall\,t\in [-T,T].
\ee

\end{thm}
Next, we introduce the notation :  given $f:\mathbb R \to\mathbb R$
for $k\in \mathbb Z, k\geq 0$ and  $\theta\in[0,1]$ define
\be\label{sin-set}
\Upsilon_{k,\theta}(f)=\{x\in\mathbb R\;:\;f\;\;\text{is not}\;\;C^{k,\theta}\;\;\text{at}\;\;x\,\}.
\ee

\begin{cor}\label{cons}
 If $u_0\in C^{k,\theta}(\mathbb R)$,  $k\in\mathbb Z,\,k\geq 0,$ and $\theta\in[0,1]$, then the corresponding  solution  $u\in C([-T,T]: C^{k,\theta}(\mathbb R))$ of the IVP \eqref{bbm} obtained  in Theorem \ref{THA1} 
 satisfies for all $t\in[-T,T]$
\be
\label{regA1}
\Upsilon_{j,\eta}(u(t))=\Upsilon_{j,\eta}(u_0),\,j\in\{k;k+1\},\,\eta\in[0,1],\;j+\eta\leq k+1+\theta.
\ee

\end{cor}

\begin{rem}
\begin{enumerate}
\item
Roughly, the reversibility in time of the equation and the above results tells us  that in the $H^s$ and $C^{k, \theta}$ settings the regularity (and singularities) of $u_0$ are preserved by the solution $u(\cdot,t)$ in its time interval of definition.

\item
The "gain" of regularity   described in Theorem \ref{TH1} and Theorem \ref{TH1b} are quite different to those known in solutions of the IVP associated for the  KdV equation \eqref{KdV}. In this case, as a consequence of Kato smoothing effect \cite{Ka} one has : if $u(x,0)=u_0(x)\in L^2(\mathbb R)$, then the corresponding solution $u=u(x,t)$ of \eqref{KdV} satisfies
$$
u(\cdot,t)\in C^{0,1/2}_{loc}(\mathbb R),\;\;\;\text{a.e.}\;\;\;t\in\mathbb R,
$$
see, \cite{Bo1}, \cite{KPV96}.
\end{enumerate}
\end{rem}

The previous results are complemented by some  found in \cite{LPS}. There,   the propagation of local regularities in $C^{k,\theta}$ of the initial data $u_0\in H^s(\mathbb R)$ to the corresponding solution $u=u(x,t)$ of the IVP \eqref{bbm} was also studied.

Finally, we observe that in \cite{ChHa}, the system of BBM equations
\be
\label{sys-bbm}
\begin{aligned}
\begin{cases}
&\partial_t u +\partial_x u -\partial^3_{xxt}u +\partial_x(Au^2+Buv+Cv^2)=0,\\
&\partial_t v +\partial_x v -\partial^3_{xxt}v +\partial_x(Du^2+Euv+Fv^2)=0,
\end{cases}
\end{aligned}
\ee
with $A,B,C,D,E,F\in\mathbb R$ was considered. Under appropriate algebraic conditions on the coefficients $A,B,C,D,E,F$, it was shown  in \cite{ChHa} that the associated IVP is globally well posed in $H^s(\mathbb R)\times H^s(\mathbb R)$
for any $s\geq 0$. It should be clear from our proofs below that the  results  stated above extend, under appropriate modifications, to the solution of the system \eqref{sys-bbm}.

The rest of this paper is ordered as follows: Section 2 contains all the necessary estimates to be used in the proofs of the new results stated in the introduction. Section 3 contains the proofs concerning the unique continuation results. Section 4 have the proof of the regularity statements. 
\section{Preliminary estimates}

In this section, we shall recall some results to be used in the proofs of the theorems stated in the introduction.
\begin{lem}\label{lem1}
 Let $p\in[1,\infty]$ and $s\geq 0$, with $s$ not an odd integer in the case $p=1$. Then 
 \begin{equation}
\label{ineq3}
\|J_x^{s} (fg)\|_p \leq c (\|f\|_{p_1}\|J_x^{s}g\|_{p_2} +\|g\|_{q_1}\|J_x^{s}f\|_{q_2}),
\end{equation}
and
\begin{equation}
\label{ineq2}
\|D_x^{s} (fg)\|_p \leq c (\|f\|_{p_1}\|D_x^{s}g\|_{p_2} +\|g\|_{q_1}\|D_x^{s}f\|_{q_2})
\end{equation}
with $1/p_1+1/p_2=1/q_1+1/q_2=1/p$.

\end{lem}

For the proof of Lemma \ref{lem1} with $p\in(1,\infty)$ we refer to  \cite{KaPo}, \cite{KPV93} and \cite{GO}. The case $r=p_1=p_2=q_1=q_2=\infty$ was established in \cite{BoLi}, see also \cite{GMN}. 
The case $p=p_2=q_2=1$ was settled in \cite{OW}.

Also, we shall use the following estimate:

\begin{prop}\label{prop1}

Let $k\in\mathbb Z,\,k\geq 0$ and $\theta\in[0,1]$. Then
 \be
\label{ineq4}
\|fg\|_{C^{k,\theta}} \leq c \;
(\|f\|_{\infty} \| g\|_{C^{k,\theta}}
 +\|f\|_{C^{k,\theta}}\|g\|_{\infty}).
\ee
and
 \be
\label{ineq5}
\|f\ast g\|_{C^{k,\theta}} \leq c \;
\|f\|_1 \| g\|_{C^{k,\theta}}.
\ee
\end{prop}

Next, we recall some properties of the Bessel potential. To simplify the exposition we restrict ourselves to the one dimensional case, $n=1$. For $s>0$ 
\be
\label{2-1}
J^{-s}f(x)=(1-\partial_x^2)^{-s/2}f(x)=(G_s\ast f)(x)
\ee
where
\be
\label{2-2}
G_s(x)=c_s\;
\int_0^{\infty}\,\frac{e^{-\frac{\pi x^2}{y}-\frac{y}{4\pi}}\;dy}
{y^{1+(1-s)/2}}
\ee

Thus, for any $p\in[1,\infty]$
\be
\label{2-3}
\begin{cases}
\begin{aligned}
(a)\; \,&\;G_s\in L^p(\mathbb R),\,\;\;\;\;\;\;\,s>1-1/p,\\
(b)\; \,&\partial_xG_s\in L^p(\mathbb R),\;\;\;\;\;s>2-1/p.\\
\end{aligned}
\end{cases}
\ee

In particular, for $s=2$ (and $n=1$), one sees that  
\be
\partial_xG_2\in H^s(\mathbb R)\cap L^{\infty}(\mathbb R),\,s\in[0,1/2)\;\;\;\;\text{and}\;\;\;G_2\in C^{0,1}(\mathbb R),
\ee
since in this case $G_2(x)=e^{-|x|}/2\,$ and $\,\partial_xG_2(x)=- \sgn(x) e^{-|x|}/2$ .

Gathering this information one gets:

\begin{lem}\label{ineq1}
Let 
\be
\label{ope-phi}
\phi(D_x)=\partial_x(1-\partial_x^2)^{-1}=\partial_x J^{-2}.
\ee
Then for $p\in[1,\infty]$ 
 \be
 \label{key1}
 \|J_x^{\eta}\phi(D_x)f\|_p\leq c_{p,\theta}\|f\|_{q},
 \ee
 with $\eta\in [0,1-1/q+1/p)$ and $q\in[1,p]$.
 
 For $p\in(1,\infty)$, \eqref{key1} holds for $\eta=1$ and $q=p$.  Also for $k\in \Z, k\geq 0$ and $\theta\in[0,1]$ 
 \be\label{AA1}
  \|\phi(D_x)f\|_{C^{k,\theta}}\leq c_{k,\theta}\|f\|_{{C^{k,\theta}}},
 \ee
 Moreover, in all these cases, \eqref{key1} still holds with $J^{\theta}$ replaces by $D_x^{\theta}$. 
 \end{lem}
\vskip.1in
\begin{proof}[Proof of Lemma \ref{ineq1}]

We observe that for $\theta\in [0,1]$
$$
J^{\theta}\phi(D_x)f(x) =(\partial_x G_{2-\theta}\ast f)(x).
$$
 \eqref{key1} follows by combining \eqref{2-3}(b) and Young's inequality. In the case $p\in(1,\infty)$ and $\theta=1$, the estimate in \eqref{key1} is a consequence of  Mihlin-H\"ormander multiplier  theorem.  Finally,  in the one-dimensional case, $n=1$, for any $s>0$ and any $p\in[1,\infty]$
$$
D_x^s J^{-s} :L^p(\mathbb R)\to L^p(\mathbb R)
$$
is bounded, see \cite{St} Chapter VI, one concludes the proof of the lemma.
\end{proof}

\begin{rem} Combining \eqref{2-3} (a) and
\be
\label{id}
\partial_x \phi(D_x)=J^{-2}-I,
\ee
one gets that
\be
 \label{key1b}
 \|\partial_x\phi(D_x)f\|_p\leq c \|f\|_{p},\;\;\;\;\;p\in[1,\infty],
 \ee
 and
 \be
 \label{key1c}
 \|\partial_x\phi(D_x)f\|_{C^{k,\theta}}\leq c \|f\|_{C^{k,\theta}},\;\;\;\;\;k\in\mathbb Z, \,k\geq 0,\,\theta\in[0,1],
 \ee
 which will be also used on the next section.
\end{rem}

In the periodic case, one has that for  $\,h\in L^2(\mathbb S)$, then $$
\phi(D_x)h(x)=\partial_x(1-\partial_x^2)^{-1}h(x)= (\partial_x \mathcal G\ast h)(x)
$$
where
\begin{equation}
\label{green}
\mathcal G(x)=\frac{ \cosh(x-\lfloor x\rfloor-1/2)}{2\,\sinh(1/2)},\;\;\;\;x\in\R/\Z\cong \mathbb S,
\end{equation}
and $\lfloor \cdot\rfloor\,$ denotes the greatest integer function. 

Thus, here $\,\mathcal G(x)\,$ plays the role of the (Green) function $\,G_2(x)=e^{-|x|}/2$ on the line (for $\,(1-\partial_x^2)$). All the estimates in \eqref{2-3}-\eqref{key1c} still holds with $\,\mathcal G(x)\,$ instead of $G_2$.

Next, we recall the following result and its proof (since we shall use some modifications of it) obtained in  \cite{LiPo}. This  will be a key ingredient in the proof of our unique continuation results.

\begin{lem}[\cite{LiPo}]\label{key-uc}

(a) Let $f\in L^p(\mathbb R)$ with $p\in [1,\infty]$ such that there exists a non-empty interval $[a,b]$ such that 
\be
\label{assumption}
 f(x)\leq 0\;\;a.e.\;\;x\in [a,b]\;\;\;\text{and}\;\;\;f(x)\geq 0, \;\;a.e.\;\;x\in [a,b]^c.
\ee
Then the function
\be
\label{capF}
F(x)=\phi(D_x)f(x)=\partial_x(1-\partial_x^2)^{-1}f(x)=-\text{sgn}(\cdot)\,\frac{e^{-|\cdot|}}{2}\ast f(x)
\ee
satisfies that $F(b)\geq F(a)$ with the equality holding only when $f(x)=0, \,a.e.\;x\in\mathbb R$.

(b) The results in part (a) still hold in the periodic setting, i.e.  for $f\in L^p(\mathbb S)$.
\end{lem}

\begin{rem}

We observe that the value of $F$ defined in \eqref{capF} does not change if one replaces in its definition $f$ by $f+\alpha$ with $\,\alpha\in\mathbb R$.
\end{rem}
\begin{proof}

(a) By hypothesis $F$ is continuous so its evaluation makes sense at any point. Now, we observe that
\be
\label{0006}
-\sgn(a-y)e^{-|a-y|}< -\sgn(b-y)e^{-|b-y|}\;\;\;\;\;\text{if}\;\;\;\;y\notin (a,b),
\ee
and
\be
\label{006b}
-\sgn(a-y)e^{-|a-y|}> -\sgn(b-y)e^{-|b-y|}\;\;\;\;\;\text{if}\;\;\;\;y\in (a,b).
\ee

This combined with the definition of $F$ in  \eqref{capF} yields the desired result.

(b) It follows by just checking that if $0<a<b<1$, then
\begin{equation}
\label{007}
\partial_x\mathcal G(b-y)>\partial_x\mathcal G(a-y),\;\;\;\;\;y\in[0,1]-[a,b].
\end{equation}
and
\begin{equation}
\label{007a}
\partial_x\mathcal G(b-y)<\partial_x\mathcal G(a-y),\;\;\;\;\;y\in[a,b],
\end{equation}
\end{proof}
with
\be
\label{007c}
\partial_x\mathcal G(x)=\frac{ \sinh(x-\lfloor x\rfloor-1/2)}{2\,\sinh(1/2)}.
\ee
\section{Proof of the Unique Continuation Results}

\vskip.1in
\begin{proof}[Proof of Theorem \ref{THA1}]
We recall the notation in \eqref{ope-phi}
$$
\phi(D_x)=\partial_x(1-\partial^2)^{-1}=\partial_x J^{-2}.
$$
Hence,
\be
\label{eq-facil}
\partial_tu+\phi(D_x)\big(u+\frac{u^2}{2}\big)=0,
\ee
and
\be
\label{facil2}
u(t)-u_0=\int_0^t\phi(D_x)\big(u+\frac{u^2}{2}\big)(t')dt'.
\ee

Thus, for $u_0\in C^{k,\theta}(\mathbb R)$ we define the operator
\be
\label{defi1}
\Psi_{u_0}v(x,t)=u_0(x)+\int_0^t\phi(D_x)\big(v+\frac{v^2}{2}\big)(t')dt'
\ee
on
\be
\label{defi2}
X_T^{k,\theta}(\mathbb R)=\{ v:\mathbb R\times[-T,T]\to \mathbb R\,:\,v\in C([-T,T]:C^{k,\theta}(\mathbb R))\}.
\ee

A combination of  \eqref{ineq4}, \eqref{ineq5} and \eqref{key1c} leads to the estimates
$$
\sup_{[-T,T]}\|\Psi_{u_0}v\|_{C^{k,\theta}}\leq c\,\big(\|u_0\|_{C^{k,\theta}} +T \sup_{[-T,T]}\|v\|_{C^{k,\theta}}\big(1+\sup_{[-T,T]}\|v\|_{C^{k,\theta}}\big)\big).
$$
and
$$
\sup_{[-T,T]}\|\Psi_{u_0}(v-w)\|_{C^{k,\theta}}\leq c\, T \sup_{[-T,T]}\|v-w\|_{C^{k,\theta}}\big(1+\sup_{[-T,T]}(\|v\|_{C^{k,\theta}}+\|w\|_{C^{k,\theta}})\big).
$$

Once that the above estimates are available by restricting the size of the time interval $[-T,T]$ with $T=T\big(|u_0\|_{C^{k,\theta}}\big)>0$ one can apply the contraction principle to the operator 
$\Psi_{u_0}$ in the domain $X_T^{k,\theta}(\mathbb R)$ to obtain the desired solution $u(x,t)$ with
\be
\label{result1}
u(x,t)=u_0(x)+\int_0^t\phi(D_x)\big(u+\frac{u^2}{2}\big)(t')dt',\;\;\;\;t\in[-T,T].
\ee

To get the regularity of the solution $u(x,t)$ stated in \eqref{class-holder} we take the time derivatives in \eqref{result1} and use Proposition \ref{prop1} 
to see that  $u\in C^1([-T,T]:C^{k,\theta}(\mathbb R))$. Next, we use the equation to deduce that
\be
\label{tt}
\begin{aligned}
\partial^2_{tt}u&=-\phi(D_x)(\partial_tu+u\partial_tu)\\
&=\phi(D_x)\Big(\phi(D_x)(u+u^2/2)+u\,\phi(D_x)(u+u^2/2)\Big).
\end{aligned}
\ee
Hence, Proposition \ref{prop1} shows that $u\in C^2([-T,T]: C^{k,\theta}(\mathbb R))$. An iteration of this argument yields  the desired result.

\end{proof}

\vskip.1in
\begin{proof}[Proof of Theorem \ref{THA2}]

The idea is to reduce the proof to that given for Lemma \ref{key-uc}. 

The equation in \eqref{bbm} holds pointwise, therefore evaluating it at $(a,t_0)$ and $(b,t_0)$, substracting these values, using the remark after Lemma \ref{key-uc} and the hypothesis one gets
\be
\label{cal1}
\begin{aligned}
&0=\partial_tu(b,t_0)-\partial_tu(a,t_0) \\
&+\int_{-\infty}^{\infty}\frac{ -\sgn(b-y)e^{-|b-y|}+\sgn(a-y)e^{-|a-y|}}{2}(u+\frac{u^2}{2})(y,t_0)dy\\
&=\partial_tu(b,t_0)-\partial_tu(a,t_0) \\
&+\int_{-\infty}^{\infty}\frac{-\sgn(b-y)e^{-|b-y|}+\sgn(a-y)e^{-|a-y|}}{4}(u+1)^2(y,t_0)dy\\
&=\partial_tu(b,t_0)-\partial_tu(a,t_0)\\
& +\int_{-\infty}^a(-e^{-|b-y|}+e^{-|a-y|})\frac{(u+1)^2}{4}(y,t_0)dy\\
& +\int_{a}^b(-e^{-|b-y|}-e^{-|a-y|})\frac{(u+1)^2}{4}(y,t_0)dy\\
& +\int^{\infty}_b(e^{-|b-y|}-e^{-|a-y|})\frac{(u+1)^2}{4}(y,t_0)dy\\
&=A_1+A_2+A_3+A_4.
\end{aligned}
\ee

By hypothesis $A_1\geq 0$ and $A_3=0$. Since the integrands in $A_2$ and $A_4$ are non-negative, see \eqref{0006},  the equality in \eqref{cal1} only holds if $A_1=A_2=A_4=0$. But this implies that
$$
u(x,t_0)\equiv -1\;\;\;\;\;\forall x\,\in\R\;\;\;\;\;\text{and}\;\;\;\;\;\;\partial_tu(b,t_0)=\partial_tu(a,t_0).
$$
\end{proof}
\vskip.05in
\begin{proof}{Proof of Corollary \ref{noL2}}
It follows directly from the proof of Theorem \ref{THA2} given above.
\end{proof}

\vskip.05in

\begin{proof}[Proof of Theorem \ref{THA3}]

We look for $u_0\in C^{\infty}(\R)-\{0\}$ such that 
$$
u_0(x)=c_0\neq -1,\;\;\;\;\;\;\partial_tu(x,0)=0,\;\;\;x\in [a,b],\;\;\;a<b.
$$

Fixing $a=0$ and recalling that $f(x)=x+x^2/2$  one has that if $x\in[0,b]$
\be
\label{beg1}
\begin{aligned}
&\partial_tu(x,0)= -\phi(D_x)f(u_0(x))=\int_{-\infty}^{\infty}e^{-|x-y|}\sgn(x-y)f(u_0(y))dy\\
&=e^{-x}\int_{-\infty}^0e^yf(u_0(y))dy+e^{-x}f(c_0)\int_0^xe^ydy\\
&\;\;-e^xf(c_0)\int_x^be^{-y}dy-e^x\int_b^{\infty}e^{-y}f(u_0(y))dy\\
&=e^{-x}(\alpha+f(c_0)(e^x-1))-e^x(\beta+f(c_0)(e^{-x}-e^{-b}))\\
&=e^{-x}(\alpha-f(c_0))+e^x(e^{-b}f(c_0)-\beta),
\end{aligned}
\ee
where
\be
\label{beg2}
\alpha=\int_{-\infty}^0e^yf(u_0(y))dy,\;\;\;\;\beta=\int_b^{\infty}e^{-y}f(u_0(y))dy.
\ee

Since $c_0\neq -1$, then $f(c_0)\in (-1/2,\infty)$  with $f(-1)=-1/2$ so 
$$
\alpha\in (-1/2,\infty)\;\;\;\;\;\;\;\;\;\;\;\beta\in\big(\frac{-e^{-b}}{2},\infty\big).
$$
Thus, for any fixed $b>0$, we can find $u_0\in C^{\infty}(\R)$ such that $u_0(0)=u_0(b)=c_0$ and $u_0(x)=0$ if $|x|\geq M>0$ for large enough $M=M(c_0)>0$. This finishes the proof of Theorem \ref{THA3}.
\end{proof}

\vskip.1in
\begin{proof}[Proof of Theorem \ref{THA4}]

Using the function $f(x)=x+x^2/2$ and the remark after Lemma \ref{key-uc} we write the equation in \eqref{bbm} as
\be
\label{minus-constant}
\partial_tu+\partial_x(1-\partial_x^2)^{-1} \big(f(u)-f(c_0)\big)=0.
\ee

At this point the proof follows the argument given in the proof of Theorem \ref{THA3}.

\end{proof}

\begin{proof}[Proof of Theorem \ref{THAP2}]

It follows a similar argument to that provided in details in the proof of Theorem \ref{THA3} but using \eqref{007} and \eqref{007a} instead of 
\eqref{0006} and \eqref{006b}.

\end{proof}

\vskip.1in
\begin{proof}[Proof of Theorem \ref{THAP3}]

We re-write the equation in \eqref{bbm-p} as
\be
\label{abc1}
\partial_tu+\phi(D_x)\big(f(u)-f(c_0)\big)=0,
\ee
with $f(x)=x+x^2/2$.

Let $a, b\in [0,1)\cong \mathbb S$ with $0<a<b<1$. Since $c_0\neq -1$ we can find $\psi_1,\,\psi_2\in C^{\infty}_0(\mathbb S)-\{0\}$ with $\text{supp}\, \psi_1\subset(0,a)$ and $\text{supp}\, \psi_2\subset (b,1)$ such that 
$$
\aligned
&f(\psi_1(x))< f(c_0),\;\;\;\;x\in (0,a),\\
&f(\psi_2(x))> f(c_0),\;\;\;\;x\in (b,1).
\endaligned
$$

Let
$$
v_j(x)=c_0+\psi_j(x),\;\;\;\;j=1,2.
$$
and define the operator
$$
Q:H^s(\mathbb S)\to \R,\;\;\;\;s>1/2,
$$
as
\be
\label{abc2}
\begin{aligned}
Q(u_0)\equiv &\partial_tu(b,0)-\partial_tu(a,0)\\
=-&\int_0^1 \big(\partial_x\mathcal G(b-y)-\partial_x\mathcal G(a,0)\big)(f(u(x,0))-f(c_0))dy
\end{aligned}
\ee
where $u\in C^{\infty}((-T,T):H^s(\mathbb S))$ is the solution of the IVP \eqref{bbm} provided by Theorem \ref{RoZh}.

Hence,  from \eqref{007}, \eqref{007a} and \eqref{007c} one has that
\be
\label{abc5}
Q(v_1)>0\;\;\;\;\;\text{and}\;\;\;\;\;Q(v_2)<0.
\ee

The operator $ Q$ defined in \eqref{abc2} is continuos. Therefore, by  considering
$$
Q(\lambda v_1+(1-\lambda)v_2),\;\;\;\;\;\;\lambda\in [0,1],
$$
one has that there exists $\lambda_0\in(0,1)$ such that for $v_0=\lambda_0v_1+(1-\lambda_0)v_2$ one has $Q(v_0)=0$, i.e.
\be
\label{abc3}
 \partial_tv(b,0)-\partial_tv(a,0)=0,
 \ee
 with $v(x,t)$ the solution corresponding to the data $v_0$ which satisfies $v_0(x)=c_0,\;x\in [a,b]$. This yields the desired result.

\end{proof}

\vskip.1in
\begin{proof}[Proof of Theorem  \ref{IVPCH2}]

It follows an  argument similar to that provided in details in the proof of Theorem \ref{THA2}. Hence, it will be omitted.
\end{proof}

\vskip.1in
\begin{proof}[Proof of Theorem  \ref{THAP4}]

It follows a similar argument to that provided in details in the proof of Theorem \ref{THA3} but using \eqref{007} and \eqref{007a} instead of 
\eqref{0006} and \eqref{006b}.

\end{proof}
\section{Proof of the Regularity Results}

\begin{proof}[Proof of Theorem \ref{TH1}]

We recall the notation in \eqref{ope-phi}
$$
\phi(D_x)=\partial_x(1-\partial^2)^{-1}=\partial_x J^{-2}.
$$
Hence,
\be
\label{eq-facil1}
\partial_tu+\phi(D_x)\big(u+\frac{u^2}{2}\big)=0,
\ee
and
\be
\label{facil21}
u(t)-u_0=\int_0^t\phi(D_x)\big(u+\frac{u^2}{2}\big)(t')dt'.
\ee

First, we consider the case $s\in[0,1/2)$.

Applying the operator $J_x^{(2s+1/2)^{-}} $ to the integral term in \eqref{facil21}, and combining Proposition \ref{ineq2},  \eqref{key1} with $(\eta,p,q)=((s+1/2)^-,2,2)$ and  $(\eta,p,q)=((1/r)^-,2,q)$ with $1+1/2=1/r+1/q$ and $s+1/2=1/r+s_2$, Lemma \ref{lem1} and Sobolev embedding theorem, one gets
\be
\label{E1}
\begin{aligned}
\| J^{(2s+1/2)^-}_x(u&(t)-u_0)\|_2\\
&\leq c\int_0^t (\|J_x^s u(t')\|_2+\|J_x^{s_2}(u^2)(t')\|_q)dt'\\
&\leq c\int_0^t (\|J_x^s u(t')\|_2+\|u(t')\|_r\|J_x^{s_2} u(t')\|_l)dt'\\
&\leq c\int_0^t (\|J^{s}_xu(t')\|_2+\|J^{s}_xu(t')\|^2_{2})dt',\\
\end{aligned}
\ee
with $1/q=1/r+1/l$, 
$
s-1/2=-1/r$ and $s-1/2=s_2-1/l$ such that
$$
1/r+s_2=1/r+2s+1/q-1=2s+1/2.
$$

This yields the desired result \eqref{m5}.

Next, we consider the case in \eqref{m6}, i.e.  $u_0\in H^{s}(\mathbb R)$ with $s>1/2$.

Applying the operator $J_x^{(2s+1)} $ to the integral term in \eqref{facil21}, and combining \eqref{key1} with $(\eta,p,q)=(1,2,2)$,  Lemma \ref{ineq1} and Sobolev embedding theorem, one gets
\be
\label{E2}
\begin{aligned}
\| J^{s+1}_x(u&(t)-u_0)\|_2\leq c\int_0^t (\|J_x^s u(t')\|_2+\|J_x^{s}(u^2)(t')\|_2)dt'\\
&\leq c\int_0^t (\|J_x^s u(t')\|_2+\|u(t')\|_{\infty}\|J_x^{s} u(t')\|_2)dt'\\
&\leq c\int_0^t (\|J^{s}_xu(t')\|_2+\|J^{s}_xu(t')\|^2_{2})dt',\\
\end{aligned}
\ee
which yields  \eqref{m6}.

To prove the regularity in the $t$-variable, we first show that if
$u, v\in C([-T,T]:H^{s}(\mathbb R))$, then 
\be
\label{reg1}
\begin{aligned}
&(i)\;\;\phi(D_x) u\in C([-T,T]:H^{s}(\mathbb R)),\\
&(ii)\;\phi(D_x)(uv)\in C([-T,T]:H^{s}(\mathbb R)).
\end{aligned}
\ee
We recall that
$$
\phi(D_x) f(x)=\partial_xG_2\ast f(x),
$$
with $\partial_xG_2(x)=-\sgn(x)e^{-|x|}/2\in L^1(\R)\cap L^2(\R)$. Thus
$$
\|J_x^s\phi(D_x) u\|_2\leq \|J_x^su\|_2
$$
and using Lemma \ref{lem1}
$$
\|J_x^s\phi(D_x) (uv)\|_2= \|\phi(D_x) J_x^s(uv)\|_2\leq \|J_x^s(uv)\|_1\leq c\|J^su\|_2\|J^sv\|_2
$$
which yields \eqref{reg1}.

 Since $u\in C([-T,T]:H^{s}(\mathbb R))$ the equation in \eqref{facil21} and \eqref{reg1}  shows that  $u\in C^1([-T,T]:H^{s}(\mathbb R))$. Next, we use \eqref{eq-facil1} to deduce that
\be
\label{ttt}
\begin{aligned}
\partial^2_{tt}u&=-\phi(D_x)(\partial_tu+u\partial_tu)\\
&=\phi(D_x)\Big(\phi(D_x)(u+u^2/2)+u\,\phi(D_x)(u+u^2/2)\Big).
\end{aligned}
\ee
Applying \eqref{reg1}  one gets that $u\in C^2([-T,T]: H^{s}(\mathbb R))$. An iteration of this argument yields  the proof of Theorem \ref{TH1} part(i).

To prove part (ii) one just needs to show that 
$$
\partial_x^{k+1}(u(x,t)-u_0(x))\in C^{0,\theta}(\mathbb R).
$$

Applying the operator $\partial_x^{k+1} $ to the integral term in \eqref{facil21}, and combining \eqref{key1c} and Proposition \ref{prop1} one gets
\be
\label{E3}
\begin{aligned}
\| \partial_x^{k+1}(u&(t)-u_0)\|_{C^{0,\theta}(\mathbb R)}\\
&\leq c\int_0^t (\|\partial_x^k u(t')\|_{C^{0,\theta}(\mathbb R)}+\|u^2(t')\|_{C^{k,\theta}(\mathbb R)})dt'\\
&\leq c\int_0^t (\|u(t')\|_{C^{k,\theta}(\mathbb R)}+\|u(t')\|_{\infty}\|u(t')\|_{C^{k,\theta}(\mathbb R)})dt'\\
\end{aligned}
\ee
which basically yields  \eqref{m5a}. This completes the proof of Theorem \ref{TH1}.

\end{proof}
\begin{proof}[Proof of Theorem \ref{TH1b}]

The proof is similar to that provided in details for Theorem \ref{TH1}. Therefore, it will be omitted.
\end{proof}

\begin{proof}[Proof of Theorem \ref{TH1c}]
By writing the solution $u=u(x,t)$ as
\be
\label{key-idea}
u(x,t)=(u(x,t)-u_0(x)) + u_0(x),
\ee
with $u_0\in H^s(\R)$, from  Theorem \ref{TH1} one has 
$$
u-u_0\in C^{\infty}([-T,T]:H^{s'}(\R)),
$$
with
\be
\label{best}
s'=
\begin{cases}
\begin{aligned}
& (2s+1/2)^-,\;s\in[0,1/2],\\
&\;\;\;s+1,\;\;\;\;\;\;\;\,\;s>1/2.
\end{aligned}
\end{cases}
\ee
Therefore, the regularity (and singularities) of $u_0$ at the level of $H^{\mu}(\Omega)$, with $\mu\in(s,s']$, are the same as those of $u=u(x,t)$ for all $t\in[-T,T]$.

\end{proof}

\begin{proof}[Proof of Theorem \ref{TH1d}]
The proof is similar to that provided in details for Theorem \ref{TH1}. Therefore, it will be omitted.

\end{proof}

\begin{proof}[Proof of Corollary \ref{cons}]

Combining Sobolev embedding theorem, the expression in \eqref{key-idea} and the results obtained in Theorem \ref{TH1}, one gets the desired result.

\end{proof}
\section{Declarations}

\subsection*{Author Contributions} Applicable for submissions with multiple authors): the authors contribute equally to
the manuscript.

\subsection*{Data Availibility} A statement on how any datasets used can be accessed: not applicable.

\subsection*{Conflict of interest} Always applicable and includes interests of a financial or personal nature:
the authors declare no conflict of interest.
\subsection*{Ethical Approval} Applicable for both human and/or animal studies. Ethical committees, Internal Review
Boards and guidelines followed must be named. When applicable, additional headings with statements on
consent to participate and consent to publish are also required: not applicable.

\end{document}